\documentclass{amsart}

\usepackage[title]{appendix}
\usepackage{amssymb}
\usepackage{amsmath}

\usepackage{amsthm}
\usepackage{tikz-cd}
\usepackage{mathtools}

\usepackage{framed}
\usepackage{comment}

\newcommand{\C}{\mathbb{C}}

\newcommand{\GL}{\text{GL}}

\newcommand{\Q}{\mathbb{Q}}
\newcommand{\R}{\mathbb{R}}

\newcommand{\Z}{\mathbb{Z}}
\renewcommand{\H}{\mathbb{H}}
\newcommand{\F}{\mathbb{F}}

\DeclareMathOperator{\Aut}{Aut}

\DeclareMathOperator{\rank}{rank}

\DeclareMathOperator{\LF}{LF}
\DeclareMathOperator{\im}{Im}

\DeclareMathOperator{\coker}{coker}
\DeclareMathOperator{\Wh}{Wh}
\DeclareMathOperator{\tors}{tors}
\DeclareMathOperator{\rk}{rk}

\newtheorem{thm}{Theorem}[section]
\newtheorem{prop}[thm]{Proposition}
\newtheorem{lem}[thm]{Lemma}

\newtheorem*{D2}{D2 problem}
\newtheorem{question}[thm]{Question}
\theoremstyle{definition}

\theoremstyle{remark}
\newtheorem{remark}[thm]{Remark}

\newcounter{tmp}

\usepackage{enumerate}
\usepackage{rotating}

\usepackage{mathrsfs}

\makeatletter
\@namedef{subjclassname@2020}{%
  \textup{2020} Mathematics Subject Classification}
\makeatother

\begin{document}

\title{A cancellation theorem for modules over integral group rings}

\author{John Nicholson}
\address{Department of Mathematics, UCL, Gower Street, London, WC1E 6BT, U.K.}
\email{j.k.nicholson@ucl.ac.uk}



\begin{abstract}
A long standing problem, which has its roots in low-dimensional homotopy theory, is to classify all finite groups $G$ for which the integral group ring $\Z G$ has stably free cancellation (SFC).
We extend results of R. G. Swan by giving a condition for SFC and use this to show that $\Z G$ has SFC provided at most one copy of the quaternions $\H$ occurs in the Wedderburn decomposition of the real group ring $\R G$.\,\,This generalises the Eichler condition in the case of integral group rings.
\vspace{-5mm}
\end{abstract}

\maketitle


\section*{Introduction}

A ring $R$ is said to have \textit{stably free cancellation} (SFC) if $M \oplus R^n \cong R^{n+m}$ implies $M \cong R^m$ for finitely generated modules $M$ or, equivalently, if $R$ has no non-trivial stably free modules.
In this paper, we will be interested in the problem of determining which finite groups $G$ have the property that the integral group ring $\Z G$ has SFC. 

Fix a finite group $G$ once and for all, and recall that the real group ring $\R G$ is semisimple and so admits a Wedderburn decomposition
\[ \R G \cong M_{n_1}(D_1) \times \cdots \times M_{n_r}(D_r)\]
where $n_1, \cdots, n_r$ are integers $\ge 1$ and each $D_i$ is one of the real division algebras $\R$, $\C$ or $\H$. We say that $\Z G$ satisfies the \textit{Eichler condition} if $D_i \ne \H$ whenever $n_i=1$. By the Jacobinski cancellation theorem \cite{Sw70}, this is a sufficient condition for $\Z G$ to have SFC.

It is well-known, and is discussed in Section 1, that $\Z G$ fails the Eichler condition precisely when $G$ has a quotient which is a binary polyhedral group, i.e. a non-cyclic finite subgroup of $\H^\times$. They are the generalised quaternion groups $Q_{4n}$ for $n \ge 2$ or one of $\widetilde{T}, \widetilde{O}, \widetilde{I}$, the binary tetrahedral, octahedral and icosahedral groups.
It was shown by Swan \cite{Sw83} that, if $G$ is a binary polyhedral group, then $\Z G$ has SFC if and only if $G$ is not of the form $Q_{4n}$ for $n \ge 6$, i.e. if $G$ is one of the groups
\[\tag{*} Q_8,Q_{12},Q_{16},Q_{20},\widetilde{T},\widetilde{O},\widetilde{I}.\] 

It follows from a result of Swan-Fr\"{o}hlich \cite[Theorem A10]{Sw83} that, if $G$ has a quotient $H$ and $\Z G$ has SFC, then $\Z H$ has SFC also. 
In particular, $\Z G$ does not have SFC whenever $G$ has a quotient which is $Q_{4n}$ for $n \ge 6$. Note that this does not yet characterise which groups have SFC; it remains to determine SFC for $\Z G$ when $G$ has a quotient in one of the groups in $(*)$ but none in $Q_{4n}$ for $n \ge 6$.

Some results in this direction were obtained by Swan where he showed that, if $G$ is one of the groups in $(*)$ and $H$ satisfies the Eichler condition, then $\Z[G \times H]$ has SFC provided $H$ has no normal subgroup of index 2.	
He also proved that $\Z[\widetilde{T}^n \times \widetilde{I}^m]$ has SFC for all $n, m \ge 0$ and that $\Z[Q_8 \times C_2]$ does not have SFC.

The main aim of this paper is to extend these results to a large class of new groups, including many which do not split as direct products. 

Define the \textit{$\H$-multiplicity} $m_{\H}(G)$ to be the number of copies of $\H$ in the Wedderburn decomposition of $\R G$, i.e. the number of irreducible one-dimensional quaternionic representations, so that $G$ satisfies the Eichler condition if and only if $m_{\H}(G)=0$.
Recall also from algebraic K-theory that
$K_1(R)=\GL(R)^{ab}$
where $\GL(R) = \bigcup_n \GL_n(R)$ with respect to the inclusions $\GL_n(R) \hookrightarrow \GL_{n+1}(R)$.
We will say $K_1(R)$ is \textit{represented by units} if the natural map $R^\times = \GL_1(R) \hookrightarrow \GL(R) \to K_1(R)$ is surjective.

\begingroup
\setcounter{tmp}{\value{thm}}
\setcounter{thm}{0} 
\renewcommand\thethm{\Alph{thm}}
\begin{thm} \label{thm:A}
Suppose $G$ has a quotient $H$ such that $m_{\H}(G)=m_{\H}(H)$ and $K_1(\Z H)$ is represented by units. Then $\Z G$ has SFC if and only if $\Z H$ has SFC.
\end{thm}
\endgroup

This can be seen as a relative version of the Eichler condition and gives a relation between SFC and the classical problem of determining when $K_1(\Z H)$ is represented by units \cite{JM80}. Whilst unit representation fails in general, $H=Q_{116}$ being an example, we will deduce the following in Section \ref{section:unit-rep-quat-rep} as an application of results in \cite{MOV83}.

\begingroup
\setcounter{tmp}{\value{thm}}
\renewcommand\thethm{\Alph{thm}}
\begin{thm} \label{thm:B}
Suppose $G$ has a binary polyhedral quotient $H$ such that $m_{\H}(G)=m_{\H}(H)$. Then $\Z G$ has SFC if and only if $\Z H$ has SFC.
\end{thm}
\endgroup

This recovers Swan's result for direct products and also proves SFC for infinitely many new groups, the smallest example being the non-trivial semidirect product $C_4 \rtimes C_4$.
We also prove the following generalisation of the Eichler condition to integral group rings.

\begingroup
\setcounter{tmp}{\value{thm}}
\renewcommand\thethm{\Alph{thm}}
\begin{thm} \label{thm:C}
If $m_{\H}(G) \le 1$, then $\Z G$ has SFC.	
\end{thm}
\endgroup

In Section \ref{section:towards-full-classification}, we conclude by showing how close these results bring us to a full classification of the finite groups $G$ for which $\Z G$ has SFC. In particular, we formulate a precise question regarding the complexity that the final classification might take.

We will end the introduction with a brief discussion of the motivation for this work, and of some of its applications to topology. Recall the following question of C. T. C. Wall \cite{Wa65}, which remains a significant open problem:

\begin{D2}
Let $X$ be a finite connected CW-complex which is cohomologically $2$-dimensional, i.e. 
\[H_i(X;M)=H^i(X;M)=0\] 
for all $i > 2$ and all $\Z [\pi_1(X)]$ modules $M$. Then is $X$ homotopic to a finite $2$-complex?
\end{D2}

This can be viewed as being parametrised by groups $G$ by saying that $G$ has the D2 property if the answer is yes for all finite complexes $X$ with fundamental group $G$. The D2-property has been actively studied for finite groups but has so far only been proven in a limited number of cases

Results of F. E. A. Johnson \cite{Jo03} have made explicit the link between SFC and the D2 property for groups with $4$-periodic cohomology. For example, if $X$ is a cohomologically $2$-dimensional finite complex such that $\pi_1(X)=G$ admits a free resolution of period $4$, then 
$ \pi_2(X) \oplus \Z G^r \cong I^* \oplus \Z G^s$
for some $s \ge r \ge 0$, where $I^*$ is the dual of the augmentation ideal. It is not difficult to show that SFC implies cancellation in the class of $I$ or $I^*$ and so, in such a case, we would have $\pi_2(X) \cong I^* \oplus \Z G^{s-r}$ which places a significant restriction on the possible homotopy types of such an $X$.

In \cite{Ni19}, we apply the results of this article to show that a group $G$ with periodic cohomology has SFC if and only if $m_{\H}(G) \le 2$. This is used to show that the D2 property is equivalent to the existence of a balanced presentation for $4$-periodic groups $G$ with $m_{\H}(G) \le 2$. In addition, this sheds new light on the old problem of whether or not a Poincar\'{e} $3$-complex has a cellular structure with a single $3$-cell \cite{Wa67}.
For a recent survey on the link between SFC and problems in low-dimensional homotopy theory, see \cite[Chapter 1]{MR17}.


\section{Preliminaries} \label{section:preliminaries}

Throughout this article, for a ring $R$, all $R$ modules will be taken to be finitely generated right $R$ modules and $G$ will be taken to be a finite group.

Let $A$ be a finite-dimensional semisimple $\Q$-algebra and let $\Lambda$ be a $\Z$-order in $A$, i.e. a finitely-generated subring of $A$ such that $\Q \cdot \Lambda = A$. For example, we can take $\Lambda = \Z G$ and $A = \Q G$.
For a prime $p$, let $\Z_p$ denote the $p$-adic integers and let $\Lambda_p = \Lambda \otimes \Z_p$. For $n \ge 1$, we say that a $\Lambda$ module $M$ is \textit{locally free of rank $n$} if $M_p = M \otimes \Z_p$ is a free $\Lambda_p$ module of rank $n$ for all primes $p$.

By \cite[Lemma 2.1]{Sw80}, every locally free $\Lambda$ module is projective. The converse holds in the case where $\Lambda = \Z G$ by \cite[Theorems 2.21 and 4.2]{Sw70}, i.e. a $\Z G$ module is projective if and only if it is locally free. As was noted by Swan in \cite[p156]{Sw80}, most special properties of projective modules over $\Z G$ are consequences of this fact.

Let $\LF_n(\Lambda)$ denote the set of isomorphism classes of locally free $\Lambda$ modules of rank $n$ and let $C(\Lambda)$ denote the locally free class group, i.e. the equivalence classes of locally free $\Lambda$ modules where $P_1 \sim P_2$ if $P_1 \oplus \Lambda^n \cong P_2 \oplus \Lambda^m$ for some $n,m \ge 0$.
The map $P \mapsto P \oplus \Lambda$ induces a sequence
\[
\begin{tikzcd} 
\LF_1(\Lambda) \arrow[r,twoheadrightarrow] & \LF_2(\Lambda) \arrow[r,"\cong"] & \LF_3(\Lambda) \arrow[r,"\cong"] & \cdots \arrow[r,"\cong"] & C(\Lambda)
\end{tikzcd}
\]
where all the maps are surjections by Serre's Theorem and all but the first map are bijections by Bass' Cancellation Theorem \cite[Section 2]{Sw80}. 

For a locally free $\Lambda$ module $P$, it will often be useful to write $[P]$ to denote both the class $[P] \in C(\Lambda)$ and the set of isomorphism classes of all locally free $\Lambda$ modules $P_0$ for which $[P_0] = [P] \in C(\Lambda)$. We can view this set as a tree by placing edges between each $P' \in [P]$ and $P' \oplus \Lambda \in [P]$. By the above, this tree has the structure of a fork: a finite set of vertices for the $P' \in [P]$ of rank $1$ and a single vertex for each $P ' \in [P]$ of rank $>1$ (see Figure \ref{figure:fork} below).

\begin{figure}[h] \vspace{-2mm} \label{figure:fork}
\begin{tikzpicture}
\draw[fill=black] (0,0) circle (2pt);
\draw[fill=black] (1,0) circle (2pt);
\draw[fill=black] (2,0) circle (2pt);
\draw[fill=black] (3,0) circle (2pt);
\draw[fill=black] (4,0) circle (2pt);
\draw[fill=black] (2,1) circle (2pt);
\draw[fill=black] (2,2) circle (2pt);
\draw[fill=black] (2,3) circle (2pt);
\node at (2,3.6) {$\vdots$};
\draw[thick] (0,0) -- (2,1) (1,0) -- (2,1) (2,0) -- (2,1) (3,0) -- (2,1) (4,0) -- (2,1) -- (2,2) -- (2,3);
\end{tikzpicture}
\caption{Tree structure for $[P] \in C(\Lambda)$}
\end{figure}

Hence the stable class map $[\,\cdot\,] : \LF_1(\Lambda) \to  C(\Lambda)$ is surjective and is bijective precisely when $\Lambda$ has locally free cancellation, i.e. cancellation in the class of all locally free modules. Furthermore, $\Lambda$ has SFC precisely when the fibre over $0 \in C(\Lambda)$ is trivial.

We will now recall Swan's refinement of Fr{\"o}hlich's result \cite[III]{Fr75} that, if $G$ has a quotient $H$ and $\Z G$ has projective cancellation, then $\Z H$ has projective cancellation.
This gives one direction of Theorem \ref{thm:A}.

\begin{thm}[Swan-Fr\"{o}hlich] \label{thm:quotient-closed}
If $G$ has a quotient $H$ and $\Z G$ has SFC, then $\Z H$ has SFC.
\end{thm}

\begin{proof}
Since the map $\Z G \twoheadrightarrow \Z H$ of $\Z$-orders induces a surjection $\Q G \twoheadrightarrow \Q H$, we can apply \cite[Theorem A10]{Sw83}. In particular, this shows that
\[
\begin{tikzcd}
\LF_1(\Z G) \arrow[r] \arrow[d, twoheadrightarrow] & \LF_1(\Z H)  \arrow[d,twoheadrightarrow] \\
C(\Z G) \arrow[r] & C(\Z H)
\end{tikzcd}
\]
is a weak pullback in the sense that that compatible elements in the corners have not-necessarily-unique lifts to $\LF_1(\Z G)$. It follows that the fibre in $\LF_1(\Z G)$ over $0 \in C(\Z G)$ maps onto the fibre in $\LF_1(\Z H)$ over $0 \in C(\Z H)$.
\end{proof}

Since $\Lambda_{\R} = \Lambda \otimes \R$ has a real Wedderburn decomposition, the Eichler condition generalises to $\Z$-orders $\Lambda$ in the natural way. 
The following is a consequence of a general cancellation theorem proven by H. Jacobinski \cite[Theorem 4.1]{Ja68} which depends on a deep result of Eichler concerning strong approximation for the kernel of the reduced norm map.

\begin{thm}[Jacobinski] If $\Lambda$ satisfies the Eichler condition, then $\Lambda$ has locally free cancellation.
\end{thm}

The following will be an essential ingredient in the proof of Theorem \ref{thm:A} and can be found in \cite[Corollary 10.5]{Sw80}.

\begin{thm}[Swan] \label{thm:units-K_1}
If $\Lambda$ satisfies the Eichler condition and $I$ is a two-sided ideal of finite index then, with respect to the quotient map $f: \Lambda \twoheadrightarrow \Lambda/I$, we have $f(\Lambda^\times) \unlhd (\Lambda/I)^\times$ and the map of units $(\Lambda/I)^\times \to K_1(\Lambda/I)$ induces an isomorphism:
	\[ \frac{(\Lambda/I)^\times}{\Lambda^\times} \cong \frac{K_1(\Lambda/I)}{K_1(\Lambda)}. \]
\end{thm}

\begin{remark} 
Note that the statement that $(\Lambda/I)^\times \to K_1(\Lambda/I)$ induces this isomorphism is not stated explicitly in \cite{Sw80}. To see this, recall that the proof uses Swan's adaptation of the strong approximation theorem \cite[Theorem 8.1]{Sw80} to compare 
\[\coker(\Lambda^\times \to (\Lambda/I)^\times)\] 
to an appropriate quotient of $\nu(\widehat{\Lambda}^\times)$, where here $\nu$ denotes the reduced norm and $\widehat{\Lambda}=\widehat{\Z} \otimes_{\Z} \Lambda$ is the adelic completion. This is then identified with 
\[\coker(\GL_n(\Lambda) \to \GL_n(\Lambda/I))\] 
by showing that $\nu(\widehat{\Lambda}^\times) = \nu(\GL_n(\widehat{\Lambda}))$. Since $\nu$ is compatible with the inclusion maps $\GL_n(\widehat{\Lambda}) \hookrightarrow \GL_{n+1}(\widehat{\Lambda})$, the map between cokernels is induced by the inclusion of units. The result then follows by letting $n$ tend to  infinity. 
\end{remark}

Finally we include a proof of the following result mentioned in the introduction which reduces the Eichler condition to a purely group-theoretic property.
This is well-known \cite{Sw83} though does not appear to be written down explicitly in the literature except in the backward direction \cite[p305]{CR87}.

As before, we will use \textit{binary polyhedral group} to mean a non-cyclic finite subgroup of $\H^\times$. They are the generalised quaternion groups $Q_{4n}$ for $n \ge 2$ or one of $\widetilde{T}, \widetilde{O}, \widetilde{I}$, the binary tetrahedral, octahedral and icosahedral groups.

\begin{prop} \label{prop:eichler-group}
$\Z G$ fails the Eichler condition if and only if $G$ has a quotient which is a binary polyhedral group.	
\end{prop}

\begin{proof}
	If $G$ fails the Eichler condition, the Wedderburn decomposition gives a map $G \to \H^\times$, i.e. a quaternionic representation. Since $G$ is an $\R$-basis for $G$, the image must contain an $\R$-basis for $\H$. Since $\H$ is non-commutative, the image must be non-abelian and so a binary polyhedral group. 
	Conversely, a quotient of $G$ into a binary polyhedral group gives a representation $G \to \H^\times$ which does not split over $\R$ or $\C$ since the image is non-abelian. Hence the representation is irreducible and so represents a term in the Wedderburn decomposition.
\end{proof}


\section{Proof of Theorem A} \label{section:proof-of-main}

Let $N$ be a normal subgroup of $G$ and let $H=G/N$. Suppose $m_{\H}(G)=m_{\H}(H)$ and that $\Z H$ has SFC. Since the other direction was proven in the preceding section, it will suffice to prove that $\Z G$ has SFC subject to the conditions of Theorem A.
Consider the following pullback diagram for $\Z G$ induced by $N$:
\[
\begin{tikzcd}
  \Z G \arrow[r,"i_2"] \arrow[d,"i_1"] & \Lambda \arrow[d,"j_2"] \\
  \Z H \arrow[r,"j_1"] & (\Z/n \Z)[H]
\end{tikzcd}
\]
where $\Lambda = \Z G / (\widehat{N} \cdot \Z G)$, $\widehat{N}= \sum_{g \in N} g$ and $n=|N|$. This is the standard pullback construction for the ring $\Z G$ and trivially intersecting ideals 
\[I= \ker( \Z G \to \Z H) = I(N) \cdot \Z G \quad \text{and} \quad J= \widehat{N} \cdot \Z G,\] 
where $I(N)= \ker(\Z N \to \Z)$ is the augmentation ideal \cite[Example 42.3]{CR87}.

\begin{prop} $\Lambda$ satisfies the Eichler condition.	
\end{prop}
\begin{proof}
	Since $\R$ is flat, we can apply $\R \otimes -$ to the diagram to get another pullback diagram which induces $\R G \cong \R H \times \Lambda_{\R}$. Hence $m_{\H}(G) = m_{\H}(H)+m_{\H}(\Lambda_{\R})$
which implies that $m_{\H}(\Lambda_{\R})=0$ by the hypotheses on $G$ and $H$.
\end{proof}

Now note that $\Lambda$ is a $\Z$-order in $\Lambda_\Q$, which is a finite-dimensional semisimple $\Q$-algebra since $\Q G \cong \Q H \times \Lambda_\Q$. Hence we can apply Theorem 1.2 to get that $\Lambda$ has SFC. In particular, the maps 
\[\LF_1(\Lambda) \to C(\Lambda) \quad \text{and} \quad \LF_1(\Z H) \to C(\Z H)\]
have trivial fibres over $0 \in C(\Lambda)$ and $0 \in C(\Z H)$ respectively.

Consider the following diagram induced by the maps on locally free modules.
\[
\begin{tikzcd}
\LF_1(\Z G) \arrow[r,"\varphi_1"] \arrow[d, twoheadrightarrow] & \LF_1(\Z H) \times \LF_1(\Lambda) \arrow[d,twoheadrightarrow] \\
C(\Z G) \arrow[r,"\varphi_2"] & C(\Z H) \times C(\Lambda)
\end{tikzcd}
\]
Proving $\Z G$ has SFC now amounts to proving that the fibres of $\varphi_1$ and $\varphi_2$ over $(\Z H, \Lambda)$ are in bijection.
We will now compute each of these fibres in turn.

Firstly observe that the pullback diagram for $\Z G$ above is a Milnor square and so, by the general construction of projectives modules using a Milnor square (see, for example, \cite[Proposition 4.1]{Sw80}), the fibre $\varphi_1^{-1}(\Z H, \Lambda)$ is in one-to-one correspondence with the double coset 
\[ \Z H^\times \backslash (\Z/n\Z)[H]^\times / \Lambda^\times \]
which, by abuse of notation, we will take to mean $i_1(\Z H^\times) \backslash (\Z/n\Z)[H]^\times / i_2(\Lambda^\times)$.

Secondly, for a ring $R$, let $K_0(R)$ be the Grothendieck group of the monoid of isomorphism classes of projective $R$-modules $P(R)$, i.e. the abelian group generated by $[P]$ for $P \in P(R)$ with relations $[P_1 \oplus P_2] = [P_1] \oplus [P_2]$. By \cite[Theorem 3.3]{Mi71}, there is an exact sequence of the form
\[
\begin{tikzcd}
  K_1(\Z H) \times K_1(\Lambda) \rar 
    & K_1 ((\Z/n\Z)[H]) \arrow[r,"\partial"]  &   
  K_0(\Z G) \arrow[r] & K_0(\Z H) \times K_0(\Lambda)
\end{tikzcd}
\]
where all maps other than $\partial$ are functorial, which is a part of the Mayer-Vietoris sequence for the Milnor square above.

Since projective $\Z G$ modules are locally free, the locally free rank induces a surjection $\rk : K_0(\Z G) \to \Z$ and, by \cite[p157]{Sw80}, we have that $C(\Z G) \cong \ker(\rk)$. It is now straightforward to check that $\ker(\varphi_2) \cong \ker(K_0(\Z G) \to K_0(\Z H) \times K_0(\Lambda))$. Hence, by exactness, the fibre $\varphi_2^{-1}(\Z H, \Lambda)$ is in one-to-one correspondence with
\[ \text{Im}(K_1 ((\Z/n\Z)[H]) \to C(\Z G)) \cong \frac{K_1((\Z/n\Z)[H])}{K_1(\Z H) \times K_1(\Lambda)}. \]

Now note that, since $\Lambda$ satisfies the Eichler condition and since we have a quotient map $j_2 : \Lambda \twoheadrightarrow (\Z/n\Z)[H]$ where $(\Z/n\Z)[H]$ is a finite ring, we can apply Theorem \ref{thm:units-K_1} to get that $j_2(\Lambda^\times) \unlhd (\Z/n\Z)[H]^\times$  and that $(\Z/n\Z)[H]^\times \to K_1((\Z/n\Z)[H])$ induces an isomorphism
\[ \frac{(\Z/n\Z)[H]^\times}{\Lambda^\times} \cong \frac{K_1((\Z/n\Z)[H])}{K_1(\Lambda)}.\]

If $\mathcal{A}$ denotes this common abelian group, then this shows that the fibres 
\[\Z H^\times \backslash (\Z/n\Z)[H]^\times / \Lambda^\times \quad \text{and} \quad\frac{K_1((\Z/n\Z)[H])}{K_1(\Z H) \times K_1(\Lambda)}\] 
are in one-to-one correspondence if and only if the two maps $\Z H^\times \to \mathcal{A}$ and $K_1(\Z H) \to \mathcal{A}$ have the same images. 
Consider the following commutative diagram.
\[
\begin{tikzcd}
\Z H^\times \arrow[r] \arrow[d,"\varphi"] & (\Z/n\Z)[H]^\times \arrow[r,twoheadrightarrow] \arrow[d] & \frac{(\Z/n\Z)[H]^\times}{\Lambda^\times} \cong \mathcal{A} \arrow[d,equals] \\
K_1(\Z H) \arrow[r] & K_1((\Z/n\Z)[H]) \arrow[r,twoheadrightarrow] & \frac{K_1((\Z/n\Z)[H])}{K_1(\Lambda)} \cong \mathcal{A}
\end{tikzcd}
\]
If $\psi_1: \Z H^\times \to \mathcal{A}$ denotes the map along the top row and $\psi_2: K_1(\Z H) \to \mathcal{A}$ denotes the map along the bottom row, then commutativity shows that $\psi_1=\psi_2 \circ \varphi$. 

If $K_1(\Z H)$ is represented by units, i.e. $\varphi$ is surjective, then $\im \psi_1 = \im \psi_2$ and so fibres of $\varphi_1, \varphi_2$ over $(\Z H, \Lambda)$ are in bijection. This implies that $\Z G$ has SFC and so completes the proof of Theorem A.


\section{Representation of units and $\H$-multiplicity} \label{section:unit-rep-quat-rep}

In this section, we will show how Theorem \ref{thm:A} can be used to deduce Theorems \ref{thm:B} and \ref{thm:C}.
The question of which finite groups $G$ have the property that $K_1(\Z G)$ is represented by units was studied in detail in \cite{JM80} and \cite{MOV83}. 
The main result we will need is as follows, and can be found in \cite[Theorems 7.15 - 7.18]{MOV83}. Let $h_p = \# C(\Z[\zeta_p])$ be the class number of the $p$th cyclotomic field.

\begin{lem}[Magurn-Oliver-Vaserstein]
\mbox{}
\begin{enumerate}[\normalfont(i)]
\item If $n \ge 3$, then $K_1(\Z Q_{2^n})$ is represented by units
\item If $G = \widetilde{T}, \widetilde{O}, \widetilde{I}$, then $K_1(\Z G)$ is represented by units.
\item If $p$ is prime, then $K_1(\Z Q_{4p})$ is represented by units if and only if $h_p$ is odd.
\item If $p$ is prime with $p \equiv -1\mod 8$, then $K_1(\Z Q_{16p})$ is not represented by units.
\end{enumerate}
\end{lem}

In particular, unit representation fails for binary polyhedral groups in general, the smallest example given by the theorem above being $K_1(\Z Q_{116})$. It is, however, well-known that $\Z[\zeta_3]$ and $\Z[\zeta_5]$ are principal ideal domains and so $h_3=h_5=1$. Therefore the theorem above shows that $K_1(\Z G)$ has unit representation for all $G$ is the list from the introduction:
\[ \tag{*} Q_8, Q_{12}, Q_{16}, Q_{20}, \widetilde{T}, \widetilde{O}, \widetilde{I}. \]

\begin{proof}[Proof of Theorem \ref{thm:B}]
Suppose $G$ has a binary polyhedral quotient $H$ such that $m_{\H}(G)=m_{\H}(H)$. If $\Z H$ does not have SFC, then neither does $\Z G$ by Theorem \ref{thm:quotient-closed}. If $\Z H$ has SFC, then $H$ is one of the groups in $(*)$ by \cite{Sw83} and $K_1(\Z H)$ has unit representation by the discussion above. Hence $\Z G$ has SFC by Theorem \ref{thm:A}.	
\end{proof}

To prove Theorem \ref{thm:C}, it suffices to classify groups $G$ with $m_{\H}(G) \le 1$ and show that they satisfy the conditions of Theorem \ref{thm:B}. We begin by noting that $m_{\H}(Q_{4n}) = \lfloor n/2 \rfloor$ for $n \ge 2$, which can be found in \cite{Jo03}. For the exceptional groups, we have that $m_{\H}(\widetilde{T})=1$, $m_{\H}(\widetilde{O})=2$ and $m_{\H}(\widetilde{I})=2$ which can be deduced from the character tables for these groups along with their Frobenius-Schur indicators. Hence $G = Q_8, Q_{12}, \widetilde{T}$ are the only binary polyhedral groups with $m_{\H}(G)=1$. 

\begin{remark} This gives a nice way to restate Swan's determination of SFC over binary polyhedral groups in an alternate form: if $G$ is a binary polyhedral group, then $\Z G$ has SFC if and only if $m_{\H}(G) \le 2$.
\end{remark}

\begin{proof}[Proof of Theorem \ref{thm:C}]
If $G$ has $m_{\H}(G) = 0$, then $\Z G$ has SFC by the Eichler condition and so assume $m_{\H}(G)=1$. By Proposition \ref{prop:eichler-group}, we know that $G$ must have a quotient $H$ which is a binary polyhedral group. By lifting quaternionic representations, it is clear that $m_{\H}(H) \le m_{\H}(G)=1$ and so $m_{\H}(H)=1$. By the above, this shows that $H = Q_8, Q_{12}, \widetilde{T}$ and that $m_{\H}(G)=m_{\H}(H)=1$. Hence $\Z G$ has SFC by Theorem \ref{thm:B}. 	
\end{proof}

We make two brief remarks on Theorem \ref{thm:C} and its proof. Firstly, this is the best possible result of this form in the sense that, for every $n \ge 2$, there are groups $G_n$ and $H_n$ with $m_{\H}(G_n)=m_{\H}(H_n)=n$ for which $\Z G_n$ has SFC and $\Z H_n$ does not have SFC. In particular, we can take $G_n = \widetilde{T}^n$ for $n \ge 2$,  $H_2 = Q_8 \times C_2$ and $H_n = Q_{8n}$ for $n \ge 3$.

Secondly, we can avoid any dependence on results in \cite{MOV83} in the proof of Theorem \ref{thm:C}. In particular, the fact that $K_1(\Z G)$ is represented by units in the case $G = Q_8, Q_{12}, \widetilde{T}$ follows from a simpler argument. To see this, consider
\[SK_1(\Z G) = \ker(K_1(\Z G) \to K_1(\Q G)). \]
By \cite[Theorem 14.2 and Example 14.4]{Ol88},  $SK_1(\Z G)=1$ for $G = Q_8, Q_{12}, \widetilde{T}$. If $\Wh(G) = K_1(\Z G)/ \pm G$ denotes the Whitehead group then, by \cite{Ol88}, we have 
\[ \tors(\Wh(G)) = SK_1(\Z G) \quad \text{and} \quad \rank(\Wh(G))= r_{\R}(G)-r_{\Q}(G),\] 
where $r_{\F}(G)$ is the number of irreducible $\F$-representations of $G$. 
Looking at the rational and real Wedderburn decompositions of these groups shows that $r_{\R}(G)=r_{\Q}(G)$ in each case. Hence $\Wh(G) =1$ for $G = Q_8, Q_{12}, \widetilde{T}$ and we get the stronger result that $\pm G \twoheadrightarrow K_1(\Z G)$, i.e. that $K_1(\Z G)$ is represented by the trivial units.

We conclude this section with the following, which we will use in Section \ref{section:towards-full-classification}. This gives a group-theoretic interpretation of the condition that $m_{\H}(G)=m_{\H}(H)$.

\begin{prop} \label{prop:relative-eichler-group}
Let $N$ be a normal subgroup of $G$ and let $H = G/N$. Then $N$ is contained in all normal subgroups $N'$ for which $G/N'$ is binary polyhedral if and only if $m_{\H}(G)=m_{\H}(H)$.
\end{prop}

\begin{proof}
Firstly note that $m_{\H}(G) \ge m_{\H}(H)$ holds in general by lifting quaternionic representations. By looking at the real Wedderburn decomposition, every one-dimensional quaternionic representation of $G$ corresponds to a map $\varphi: G \to \H^\times$ such that the image contains an $\R$-basis for $\H$. In particular, $\im \varphi$ is a non-abelian finite subgroup of $\H^\times$ and so is a binary polyhedral group. 

Hence quaternionic representations of $G$ precisely correspond to lifts of representations from binary polyhedral groups. The result follows immediately since the fact that every quotient from $G$ to a binary polyhedral group factors through $H$ implies that any quaternionic representation over $G$ lifts from one over $H$.
\end{proof}


\section{Groups with Stably Free Cancellation} \label{section:towards-full-classification}

We now return to the question of determining the finite groups $G$ for which $\Z G$ has SFC. 
By Theorem \ref{thm:quotient-closed}, this is equivalent to determining the set $\mathcal{Q}$ of all finite groups $G$ for which $\Z G$ does not have SFC but such that $\Z H$ has SFC for all non-trivial quotients $H=G/N$. The classification would then be that $\Z G$ has SFC if and only if $G$ has no quotient in $\mathcal{Q}$. Note that $\mathcal{Q}$ is infinite since it contains an infinite subset of the groups $Q_{4n}$ for $n \ge 6$, though we omit groups such as $Q_{72}$ from the list since $Q_{72} \twoheadrightarrow Q_{24}$. Apart from this, $\mathcal{Q}$ must also contain $Q_8 \times C_2$ but it is not yet clear how many other groups $\mathcal{Q}$ contains.

\begin{question} \label{conj}
Does $\mathcal{Q}$ contain only finitely many groups not of the form $Q_{4n}$ for $n \ge 6$?
\end{question}

Fix a finite group $G$ which has no quotient of the form $Q_{4n}$ for $n \ge 6$, but is such that $\Z G$ does not have SFC. By Swan \cite{Sw83} and Theorem \ref{thm:quotient-closed}, 
$G$ has a normal subgroup $N$ for which $G/N$ is one of the groups in $(*)$, i.e. a binary polyhedral group not of the form $Q_{4n}$ for $n \ge 6$.
By Theorem \ref{thm:B} and Proposition \ref{prop:relative-eichler-group}, we may also assume $G$ has another normal subgroup $N'$ with $N \not \subseteq N'$, $N' \not \subseteq N$ and $G/N'$ a binary polyhedral group.

Hence $G$ maps onto $\overline{G}=G/(N \cap N')$, which has disjoint quotients onto $G/N$ and $G/N'$ in that the corresponding normal subgroups $N/(N \cap N')$ and $N'/(N \cap N')$ are disjoint. 
By the second isomorphism theorem, note that
\[N/(N \cap N') \cong (N \cdot N')/N' \unlhd G/N' \]
and so $\overline{G}$ is an extension of a group $G/N$ in $(*)$ by a normal subgroup of a group in $(*)$, and similarly for the map $\overline{G} \to G/N'$. Since there are only seven groups in $(*)$, there are only finitely many possible groups $\overline{G}$ which can arise. 

It is therefore of interest to determine the finite set $\mathcal{Q}_2$ of finite groups $G$ which admit two disjoint quotients onto groups in $(*)$, and then to try to determine the all $G \in \mathcal{Q}_2$ for which $\Z G$ has SFC.

\begin{remark} If every group of in $\mathcal{Q}_2$ did not have SFC,
then every element in $\mathcal{Q}$ not of the form $Q_{4n}$ for $n \ge 6$ would be in $\mathcal{Q}_2$ and so Question \ref{conj} would be answered in the affirmative. However this is false since $\widetilde{T}^2$ gives an example of a group of this form which has SFC \cite{Sw83}. \end{remark}

The simplest examples of groups in $\mathcal{Q}_2$ are the groups $G \times C_2$ for $G$ one of the groups in $(*)$. This does not have SFC when $G=Q_8$. For a more complicated example, let $\varphi : C_4 \to \Aut(C_5^2)$ be induced by sending the generator of $C_4$ to the map $(x,y) \mapsto (x^{-1},y^{-1})$. Then the semidirect product
\[ C_5^2 \rtimes_\varphi C_4 \]
has six disjoint quotients onto $Q_{20} \cong C_5 \rtimes C_4$ induced by the six subgroups $C_5 \le C_5^2$, and can be shown to have no further quotients which are binary polyhedral groups.
A full classification of this family of groups can be carried out using tables of group extensions, the details of which we leave for a later article.

More generally, one can consider the set $\mathcal{Q}_n$ of finite groups which we define inductively by letting $\mathcal{Q}_1$ be the binary polyhedral groups and then: 
\[ \mathcal{Q}_n(+) = \{ G \in \mathcal{Q}_n : \Z G \text{ has SFC} \}, \quad \mathcal{Q}_n(-) = \mathcal{Q}_n \setminus \mathcal{Q}_n(+)\] 
\[ \mathcal{Q}_n = \{ G : \exists N, N' \unlhd G \text{ such that } N \cap N' = 1, G/N, G/N' \in \bigcup_{ m< n} \mathcal{Q}_m(+)\} \setminus \bigcup_{m < n} \mathcal{Q}_m.\]
This recovers the definition of $\mathcal{Q}_2$ and is also finite since, as before, upper bounds on the sets can be determined computationally by solving the extension problems which arise. 

We will conclude this section by showing that, using the notation above, Question \ref{conj} can be reduced to obtaining a positive answer to the following two questions:

\begin{question} \label{question:q1}
Does $K_1(\Z G)$ have unit representation for all $G \in \bigcup_{n \ge 1} \mathcal{Q}_n(+)$?
\end{question}

\begin{question} \label{question:q2}
	Is $\bigcup_{n \ge 2} \mathcal{Q}_n(-)$ finite?
\end{question}

In order to see this, we will need the following observation.

\begin{prop}
If Question \ref{question:q1} has a positive answer, then:
\[ \mathcal{Q} \subseteq \bigcup_{n \ge 1} \mathcal{Q}_n(-) \subseteq \{Q_{4n} : n\ge 6 \} \cup \bigcup_{n \ge 2} \mathcal{Q}_n(-).\]
\end{prop}

\begin{proof}

If $G \in \mathcal{Q}$, then $G$ must have a quotient in $\mathcal{Q}_1(+)$. Pick the largest $n$ for which $G$ has a quotient $H=G/N$ in $\mathcal{Q}_n(+)$. By assumption, $K_1(\Z H)$ has unit representation and so $m_{\H}(G) > m_{\H}(H)$, otherwise $\Z G$ would have SFC by Theorem \ref{thm:A}. Therefore $G$ has a quotient $H'=G/N'$ in $\mathcal{Q}_m(+)$ for some $m \le n$ such that $N \not \subseteq N'$ and $N' \not \subseteq N$. 

Hence $G$ maps onto $\overline{G}=G/(N \cap N')$ which has disjoint quotients onto $H$ and $H'$ and so is in $\mathcal{Q}_r$ for some $r \le n$.
Suppose $G \ne \overline{G}$. Then $\overline{G}$ has SFC by the definition of $\mathcal{Q}$ and so $\overline{G} \in \mathcal{Q}_r(+)$. Since we can pick $H'$ to be a maximal quotient with the given property, we can assume that $m_{\H}(G)=m_{\H}(\overline{G})$ which would imply that $\Z G$ has SFC by the fact that $K_1(\Z \overline{G})$ has unit representation and Theorem \ref{thm:A}. This is a contradiction, so $G = \overline{G} \in \mathcal{Q}_r(+)$.
\end{proof}

Hence, if both Questions \ref{question:q1} and \ref{question:q2} have a positive answer, then $\mathcal{Q}$ contains only finitely many groups outside of the set $\{Q_{4n} : n \ge 6\}$ and so Question \ref{conj} would be answered in the affirmative.

\section*{Acknowledgements}

I would like to thank my supervisor F. E. A. Johnson for many interesting discussions on stable modules and the D2 problem. I would also like to thank the referee for helpful and detailed comments. 
This work was supported by the UK Engineering and Physical Sciences Research Council (EPSRC) grant EP/N509577/1.


\begin{thebibliography}{9}

\bibitem{CR87} C. W. Curtis, I. Reiner, \textit{Methods of Representation Theory: With Applications to Finite Groups and Orders, Volume 2.} Wiley Classics Library (1987).

\bibitem{Fr75} A. Fr\"{o}hlich, \textit{Locally free modules over arithmetic orders}, J. Reine Angew. Math. \textbf{274/275} (1975), 112--138.

\bibitem{Ja68} H. Jacobinski, \textit{Genera and decompositions of lattices over orders}, Acta Math. \textbf{121} (1968), 1--29.

\bibitem{JM80} S. Jajodia, B. Magurn, \textit{Surjective stability of units and simple homotopy type.} J. Pure Appl. Algebra \textbf{18} (1980), 45--58.

\bibitem{Jo03} F. E. A. Johnson, \textit{Stable Modules and the D(2)-Problem.} London Math. Soc. Lecture Note Ser. \textbf{301}, Cambridge University Press (2003).

\bibitem{MOV83} B. Magurn, R. Oliver, L. Vaserstein, \textit{Units in Whitehead groups of finite groups.} J. Algebra \textbf{84} (1983), 324--360.

\bibitem{MR17} W. Metzler, S. Rosebrock, \textit{Advances in Two-Dimensional Homotopy and Combinatorial Group Theory.}, London Math. Soc. Lecture Note Ser. \textbf{446}, Cambridge University Press (2017).

\bibitem{Mi71} J. Milnor, \textit{Introduction to Algebraic K-Theory.} Ann. of Math. Stud. \textbf{72} (1971).

\bibitem{Ni19} J. Nicholson, \textit{On CW-complexes over groups with periodic cohomology}, arXiv:1905.12018 (2019).

\bibitem{Ol88} R. Oliver, \textit{Whitehead Groups of Finite Groups.} London Math. Soc. Lecture Note Ser. \textbf{132}, Cambridge University Press (1988).

\bibitem{Sw70} R. G. Swan, \textit{K-theory of finite groups and orders.} Lecture Notes in Math. \textbf{149}, Springer (1970).

\bibitem{Sw83} R. G. Swan, \textit{Projective modules over binary polyhedral groups.} J. Reine Angew. Math. \textbf{342} (1983), 66--172.

\bibitem{Sw80} R. G. Swan, \textit{Strong approximation and locally free modules.} Ring Theory and Algebra III, proceedings of the third Oklahoma Conference \textbf{3} (1980), 153--223.

\bibitem{Wa65} C. T. C. Wall, \textit{Finiteness conditions for CW Complexes}, Ann. of Math. \textbf{81} (1965), 56--69.

\bibitem{Wa67} C. T. C. Wall, \textit{Poincar\'{e} Complexes: I}, Ann. of Math. (2) \textbf{86} (1967), 213--245. 

\end{thebibliography}
\end{document}